\documentclass{amsproc}
\usepackage[T1]{fontenc}
\usepackage[cp1250]{inputenc}
\usepackage[english]{babel}
\usepackage{amsmath}
\usepackage{amstext}
\usepackage{amsfonts}
\usepackage{amsthm}
\usepackage{amssymb}
\usepackage{enumerate}
\makeatletter
\@namedef{subjclassname@2010}{%
\textup{2010} Mathematics Subject Classification}
\makeatother

\newcommand{\rot}{\mathtt{root}}
\newcommand{\Chi}{\mathtt{Chi}}
\newcommand{\pa}{\mathtt{par}}
\newcommand{\Des}{\mathtt{Des}}
\newcommand{\Sl}{S_{\boldsymbol{\lambda}}}

\begin{document}
\title[Aluthge transforms of weighted shifts on directed trees]{Aluthge transforms of weighted shifts on directed trees}
\author[J.\ Trepkowski]{Jacek Trepkowski}
\address{Wydzia{\l} Matematyki i Informatyki,
Uniwersytet Jagiello\'nski, ul.\ {\L}ojasiewicza~6,
PL-30348 Krak\'ow}
\email{Jacek.Trepkowski@im.uj.edu.pl}
\subjclass[2010]{Primary 47B37; Secondary
47B20, 47B33}
\keywords{Aluthge transform, weighted shift on a directed tree, hyponormal operator, composition operator on an $L^2$-space, polar decomposition}

\numberwithin{equation}{section}
\newtheorem{thm}{Theorem}[section]
\newtheorem{lem}[thm]{Lemma}
\newtheorem{prop}[thm]{Proposition}
\newtheorem{cor}[thm]{Corollary}
\newtheorem{ob}[thm]{Observation}
\theoremstyle{definition}
\newtheorem{rem}[thm]{Remark}
\theoremstyle{definition}
\newtheorem{ex}[thm]{Example}
\theoremstyle{definition}
\newtheorem{df}[thm]{Definition}

\begin{abstract}
 Aluthge transform of a bounded operator is generalized to the case of unbounded one. A formula for the Aluthge transform of a weighted shift on a directed tree is established and it is used to construct an example of a hyponormal operator whose Aluthge transform has trivial domain. It is proven that such an example can be also constructed in the class of composition operators. It is also shown that Aluthge transform of a closed, densely defined operator is not necessarily closable.
\end{abstract}
\maketitle

\section{Introduction}

 Aluthge transform of a bounded operator $T$, introduced by Aluthge in \cite{al}, is given by the formula $\widetilde{T}=|T|^{\frac12}U|T|^{\frac12}$, where $T=U|T|$ is the polar decomposition of $T$. It turned out to have many applications, e.g. in the invariant subspace problem (cf. \cite{jkp1}). One of the most important properties of the Aluthge transform is that it transforms a $p$-hyponormal operator into a $(p+\frac12)$-hyponormal one, preserving its spectrum (cf. \cite{al}, \cite{hu}). Moreover, under some conditions, the sequence $\{\widetilde{T}^{(n)}\}$ of consecutive iterations of Aluthge transform is convergent to a normal operator (cf. \cite{jkp2}). Aluthge trasforms of operators were studied also in \cite{cjl}, \cite{fjkp}, \cite{jjs3}, \cite{jkp3}, \cite{rion}.

 A natural question is which of the above mentioned properties remain true if one considers a closed, densely defined operator $T$, which is not necessarily bounded. In this paper it is shown that Aluthge transform of such an operator may have trivial domain and need not be necessarily closed or even closable. Thus the sequence $\{\widetilde{T}^{(n)}\}$ cannot be defined. What is interesting, $\widetilde{T}$ may have trivial domain even if $T$ is a hyponormal operator, which implies in particular that Aluthge transform does not preserve hyponormality in the unbounded case. An example of such a hyponormal operator is given in this paper in the class of weighted shifts on directed trees. The construction of the example is preceded by a discussion on Aluthge transform for this class of operators.

 What is important, the directed tree used in the construction is rootless and therefore the operator in question is unitarily equivalent to a composition operator. In turn, an example of an operator whose Aluthge transform is not closable can be realized as the adjoint of a composition operator.

 Since most of the properties of Aluthge transform is preserved if one replaces $\frac12$ in its definition by any other exponents that sum up to 1 (cf. \cite{al2}, \cite{hu}), in this paper $t$-Aluthge transform is considered for any $t\in(0,1]$, according to the following definition:

 \begin{df}
  Let $T$ be a closed, densely defined operator in a Hilbert space $\mathcal{H}$, let $T=U|T|$ be its polar decomposition and let $t\in(0,1]$. Then \emph{$t$-Aluthge transform} of $T$ is given by the formula
  $\Delta_t(T)=|T|^tU|T|^{1-t}.$
 \end{df}

\section{Preliminaries}
 In what follows $\mathbb{Z}$ will denote the set of all integers and $\mathbb{Z}_+ = \{0,1,2,\ldots\}$. For any set $A$ the cardinality of $A$ will be denoted by $\# A$.

 Let $T$ be any operator in a complex Hilbert space $\mathfrak{H}$. Then $\mathcal{D}(T)$, $\mathcal{N}(T)$, $\mathcal{R}(T)$ denote the domain, the null space and the range of $T$, respectively. For any linear subspace $W$ of $\mathcal{D}(T)$ we denote by $T\restriction_W$ the restriction of $T$ to the subspace $W$. Let $\Gamma(T)\subset \mathcal{H}\times\mathcal{H}$ be the graph of $T$. If the closure of $\Gamma(T)$ in the product topology is a graph of an operator, we call this operator the closure of $T$ and denote by $\overline{T}$.

 A densely defined operator $T$ is called \emph{hyponormal}, if $\mathcal{D}(T)\subset\mathcal{D}(T^*)$ and $\|Tf\|\geq\|T^*f\|$ for every $f\in\mathcal{D}(T)$.
 \bigskip

 Let $\mathfrak{T}=(V,E)$ be a directed tree (i.e. $V$  and $E$ are the sets of vertices and edges, respectively). If $\mathfrak{T}$ has a root, we denote it by $\mathtt{root}$ and we set $V^\circ=V\setminus\{\mathtt{root}\}$. Otherwise, we set $V^\circ=V$. For any vertex $u\in V$ we put $\Chi(u)=\{v\in V\,:\,(u,v)\in E\}$. If $v\in V^\circ$, than by $\pa(v)$ we denote the only vertex $u\in V$ such that $v\in\Chi(u)$.
 \bigskip

 By $\ell^2(V)$ we understand the complex Hilbert space of functions $f:V\rightarrow \mathbb{C}$ such that $\sum_{v\in V} |f(v)|^2<\infty$, with inner product $\left<f,g\right> = \sum_{v\in V} f(v)\overline{g(v)}$, $f,g\in \ell^2(V)$. For any $u\in V$ we define $e_u\in\ell^2(V)$ as follows:
 $$e_u(v) = \begin{cases}
             1, & u=v\\
             0, & u\neq v
            \end{cases}.$$
 Obviously, $\{e_u\}_{u\in V}$ is an orthonormal basis of $\ell^2(V)$. We denote by $\mathcal{E}_V$ the linear span of $\{e_u\}_{u\in V}$.
\bigskip

 For any system $\boldsymbol{\lambda}=\{\lambda_v\}_{v\in V^\circ}$ we define operator $\Sl$ in $\ell^2(V)$ by
 \begin{equation}
  \mathcal{D}(\Sl) = \left\{ f\in\ell^2(V)\,:\,\sum_{u\in V} \left(\sum_{v\in\Chi(u)}|\lambda_v|^2\right)|f(u)|^2  <\infty\right\},
  \label{eq:DSl}
 \end{equation}
 \begin{equation}
  (\Sl f)(v) = \begin{cases}
                             \lambda_v f(\pa(v)), & v\in V^\circ\\
                             0					, & v=\rot
                        \end{cases},\qquad f\in\mathcal{D}(\Sl).
  \label{eq:Sl}
 \end{equation}
 The operator $\Sl$ is called the \emph{weighted shift} on the directed tree $\mathfrak{T}$ with the system of weights $\boldsymbol{\lambda}$.

 For any $\boldsymbol{\lambda}=\{\lambda_u\}_{u\in V^\circ}$ we will use the following notations:
 $V_{\boldsymbol{\lambda}}^+ :=\{u\in V\,:\,\Sl e_u \neq 0\}$,
 $\Chi_{\boldsymbol{\lambda}}^+(u) := \Chi(u)\cap V_{\boldsymbol{\lambda}}^+$ for any $u\in V$ and, if $U$ is any subset of $V$, then $\Chi(U):= \bigcup_{u\in U} \Chi(u)$. We also use notatins $\Chi^2(u):=\Chi(\Chi(u))$ and $\pa^2(u)=\pa(\pa(u))$.

Recall some useful properties of weighted shifts.

\begin{prop}[cf. {\cite[Propositions 3.1.2 and 3.1.3]{jjs}}]
 Let $\Sl$ be a weighted shift on a directed tree $\mathfrak{T}=(V,E)$. Then the following assertions hold:
 \begin{itemize}
  \item[(i)] $\Sl$ is a closed operator,
  \item[(ii)] $e_u\in\mathcal{D}(\Sl)$ if and only if $\sum_{v\in\Chi(u)}|\lambda_v|^2<\infty$ and in this case
   \begin{equation}
    \Sl e_u = \sum_{v\in\Chi(u)} \lambda_v e_v,\qquad \|\Sl e_u\|^2 = \sum_{v\in\Chi(u)} |\lambda_v|^2,
    \label{eq:Sleu}
   \end{equation}
  \item[(iii)] $\Sl$ is densely defined if and only if $e_u\in\mathcal{D}(\Sl)$ for every $u\in V$.
 \end{itemize}
 \label{prop:Sl}
\end{prop}

\begin{lem}
 Let $\Sl$ be a weighted shift on $\mathfrak{T}=(V,E)$. Then $\mathcal{E}:=\mathcal{E}_V\cap\mathcal{D}(\Sl)$ is a core for $\Sl$, i.e. $\overline{\Sl\restriction_{\mathcal{E}}}=\Sl$.
 \label{lem:core}
\end{lem}
\begin{proof}

 Let $f\in\mathcal{D}(\Sl)$ and let
  $U:=\{u\in V\,:\, f(u)\neq 0\}$. Since $f\in\ell^2(V)$, the set $U$ is at most countable. If $U$ is finite, then $f\in \mathcal{E}_V$. Otherwise, set $U=\{u_1,u_2,\ldots\}$ and
  \begin{equation}
   f_n := \sum_{j=1}^n f(u_j) e_{u_j}, \qquad n=1,2,\ldots
   \label{eq:fn}
  \end{equation}
  Obviously, $f_n\in \mathcal{E}_V$. Since $f\in\mathcal{D}(\Sl)$, for every $n$ we have
  \begin{equation*}
   \sum_{u\in V} \left(\sum_{v\in\Chi(u)} |\lambda_v|^2 \right) |f_n(u)|^2 \leq
   \sum_{u\in V} \left(\sum_{v\in\Chi(u)} |\lambda_v|^2 \right) |f(u)|^2<\infty
  \end{equation*}
  and therefore $f_n\in\mathcal{D}(\Sl)$ and hence also $f_n\in\mathcal{E}$.

  Using Parseval's identity, we get
  \begin{align*}
   \|f-f_n\|^2 = \sum_{u\in V} |\left< f-f_n,e_u\right>|^2= \sum_{u\in V} |f(u)-f_n(u)|^2 =\\
   = \sum_{u\in U} |f(u)-f_n(u)|^2 = \sum_{j=n+1}^\infty |f(u_j)|^2 \stackrel{n\to\infty}{\longrightarrow} 0,
  \end{align*}
  because the series $\sum_{j=1}^\infty |f(u_j)|^2$ is convergent.

  It remains only to show that $\Sl f_n \to \Sl f$. Using (\ref{eq:fn}) and (\ref{eq:Sleu}), we have
  \begin{equation*}
   \Sl f_n = \Sl \sum_{j=1}^n f(u_j)e_{u_j} =
   \sum_{j=1}^n f(u_j) \left(\sum_{v\in\Chi(u_j)} \lambda_v e_v\right) .
  \end{equation*}
  Using Parseval's identity again, from (\ref{eq:Sl}) we get
  \begin{align*}
   \allowdisplaybreaks
   \|\Sl (f-f_n)\|^2 &= \sum_{v\in V} |\left<\Sl (f-f_n),e_v\right>|^2=\\
   &= \sum_{v\in V} |(\Sl f)(v)-(\Sl f_n)(v)|^2 =\\
   &= \sum_{v\in V^\circ} |\lambda_v|^2 |f(\pa(v))-f_n(\pa(v))|^2 =\\
   &= \sum_{u\in V} \sum_{v\in\Chi(u)} |\lambda_v|^2 |f(u)-f_n(u)|^2 =\\
   &= \sum_{u\in U} |f(u)-f_n(u)|^2 \left(\sum_{v\in\Chi(u)} |\lambda_v|^2\right) = \\
   &=   \sum_{j=n+1}^\infty |f(u_j)|^2 \left(\sum_{v\in\Chi(u_j)} |\lambda_v|^2\right) \stackrel{n\to\infty}{\longrightarrow} 0,
  \end{align*}
  which completes the proof.
\end{proof}

Let us recall a criterion for hyponormality which will be used in the sequel.

\begin{thm}[cf. {\cite[Theorem 5.1.2, Remark 5.1.5]{jjs}}] Let $\Sl$ be a weighted shift with weights $\boldsymbol{\lambda}$ on a directed tree $\mathfrak{T}=(V,E)$. Then $\Sl$ is hyponormal if and only if for every $u\in V$ the following conditions hold:
 \begin{equation}
  \text{if }v\in\Chi(u)\text{ and }\|\Sl e_v\|=0\text{, then }\lambda_v=0,
   \label{eq:hyp1}
 \end{equation}
 \begin{equation}
  \sum_{v\in\Chi_{\boldsymbol{\lambda}}^+(u)} \frac{|\lambda_v|^2}{\|\Sl e_v\|^2} \leq 1.
  \label{eq:hyp}
 \end{equation}
 \label{thm:hyp}
\end{thm}

\section{Polar decompositions of $\Sl$ and $\Sl^*$}

  We begin by recalling the description of the polar decomposition of a weighted shift.

  \begin{prop}[cf. {\cite[Proposition 3.4.3]{jjs}}] Let $\Sl$ be a densely defined weighted shift on $\mathfrak{T}=(V,E)$ with weights $\boldsymbol{\lambda}$ and let  $\alpha>0$. Then:
 \begin{itemize}
  \item[(i)] $\mathcal{D}(|\Sl|^\alpha) = \{ f\in\ell^2(V)\,:\,\sum_{u\in V} \|\Sl e_u\|^{2\alpha}|f(u)|^2 <\infty \},$
  \item[(ii)] for every $u\in V$ we have $e_u\in\mathcal{D}(|\Sl|^\alpha)$ and
        $|\Sl|^\alpha e_u = \|\Sl e_u\|^\alpha e_u$,
  \item[(iii)] if $f\in\mathcal{D}(|\Sl|^\alpha)$, then
   \begin{equation}
    (|\Sl|^\alpha f)(u) = \|\Sl e_u\|^\alpha f(u),\qquad u\in V.
    \label{eq:mod}
   \end{equation}
 \end{itemize}
 \label{prop:mod}
\end{prop}

\begin{prop}[cf. {\cite[Proposition 3.5.1]{jjs}}]
 Let $\Sl$ be a densely defined weighted shift on $\mathfrak{T}=(V,E)$ with weights $\boldsymbol{\lambda}$ and let $\Sl = U|\Sl|$ be its polar decomposition. Then $U=S_{\boldsymbol{\pi}}$, where
 \begin{equation}
  \pi_v = \begin{cases}\displaystyle
                    \frac{ \lambda_v}{\|\Sl e_{\pa(v)}\|}, & \text{if }\pa(v)\in V_{\boldsymbol{\lambda}}^+\\[2ex]
                    0, & \text{otherwise} \end{cases} ,\quad v\in V^\circ.
  \label{eq:pi}
 \end{equation}
 \label{prop:Spi}
\end{prop}

The following proposition contains a formula for $\Sl^*$.

\begin{prop}[cf. {\cite[Proposition 3.4.1]{jjs}}]
 Let $\Sl$ be a densely defined weighted shift on a directed tree $\mathfrak{T}=(V,E)$. Then
 \begin{itemize}
  \item[(i)] $\mathcal{E}_V\subseteq\mathcal{D}(\Sl^*)$ and
   \begin{equation*}
    \Sl^* e_u=\begin{cases} \overline{\lambda_u}e_{\pa(u)}, &u\in V^\circ\\
    									0, &u=\rot \end{cases} ,
   \end{equation*}
  \item[(ii)] $\mathcal{D}(\Sl^*) = \left\{f\in\ell^2(V)\,:\,\sum_{u\in V}\left|\sum_{v\in\Chi(u)} \overline{\lambda_v}f(v)\right|^2<\infty\right\}$,
  \item[(iii)] $(\Sl^* f)(u) = \sum_{v\in\Chi(u)} \overline{\lambda_v}f(v)$ for every $u\in V$ and $f\in\mathcal{D}(\Sl^*)$.
 \end{itemize}
 \label{prop:Sl*}
\end{prop}

Let $\Sl^*=V|\Sl^*|$ be the polar decomposition of $\Sl^*$. From Proposition \ref{prop:Spi} it follows that $V=S_{\boldsymbol{\pi}}^*$ with $\boldsymbol{\pi}$ given by \eqref{eq:pi}. The exact formula for $S_{\boldsymbol{\pi}}^*$ can be easily derived from Proposition \ref{prop:Sl*}.

The following theorem gives a formula for powers of modulus of $\Sl^*$.

\begin{thm}
 Let $\Sl$ be a densely defined weighted shift on directed tree $\mathfrak{T}=(V,E)$ and let $\alpha>0$. Then the following assertions hold:
 \begin{itemize}
  \item[(i)] $\mathcal{D}(|\Sl^*|^\alpha) = \left\{f\in\ell^2(V)\,:\, \sum_{u\in V^\circ} \|\Sl e_u\|^{2\alpha-2} \left|\sum_{v\in\Chi(u)} \overline{\lambda_v}f(v) \right|^2 <\infty\right\}$.
  \item[(ii)] for every $u\in V$ the following formula holds:
  $$(|\Sl^*|^\alpha f)(u) = \begin{cases} \|\Sl e_{\pa(u)}\|^{\alpha-2}\lambda_u\sum_{v\in\Chi(\pa(u))}\overline{\lambda_v}f(v), &\text{if }u\in \Chi(V_{\boldsymbol{\lambda}}^+)\\
  0,& \text{if }u\in V\setminus \Chi(V_{\boldsymbol{\lambda}}^+) \end{cases},$$
  \item [(iii)] the formula
  \begin{equation}
   |\Sl^*|^\alpha = \bigoplus_{u\in V} \|\Sl e_u\|^\alpha P_u
   \label{eq:sum}
  \end{equation} holds, where $P_u$ is the orthogonal projection from $\ell^2(V)$ onto the linear span of $\Sl e_u$ for all $u\in V$ (if $\Sl e_u=0$, then $P_u=0$).
 \end{itemize}
 \label{thm:S*a}
\end{thm}

\begin{proof}
 We start by proving all assertions for $\alpha=1$.

 It is known that $\mathcal{D}(|\Sl^*|)=\mathcal{D}(\Sl^*)$, which together with (\ref{eq:DSl}), implies the part (i). Moreover, since $\Sl^*=S_{\boldsymbol{\pi}}^*|\Sl^*|$ is the polar decomposition of $\Sl^*$, where $\pi$ is given by (\ref{eq:pi}), we conclude that $|\Sl^*|=S_{\boldsymbol{\pi}} \Sl^*$. Hence for all $f\in\mathcal{D}(\Sl^*)$ and $u\in V$ we obtain
 \begin{align*}
  (|\Sl^*|f)(u) &= \begin{cases} \pi_u (\Sl^*f)(\pa(u)), &\text{if }u\in V^\circ\\[1ex]
                                 0,&\text{if }u=\rot
                                \end{cases} =\notag\\[1ex]
	&\stackrel{(\ref{eq:sum})}{=} \begin{cases} \displaystyle\frac{\lambda_u}{\|\Sl e_{\pa(u)}\|} \sum_{v\in\Chi(\pa(u))} \overline{\lambda_v} f(v),&\text{if }u\in \Chi(V_{\boldsymbol{\lambda}}^+)\\[1ex]
	         0,&\text{if }u\in V\setminus \Chi(V_{\boldsymbol{\lambda}}^+)
	        \end{cases},
 \end{align*}
 so assertion (ii) holds for $\alpha=1$.

 Let now $P_u$ be as in (iii). Then for every $f\in\ell^2(V)$, $u\in V_{\boldsymbol{\lambda}}^+$ and $w\in V$ we have
 \begin{align}
  P_u f &= \frac{ \left<f,\Sl e_u\right> \Sl e_u }{\|\Sl e_u\|^2} = \frac 1{\|\Sl e_u\|^2} \left< f,\sum_{v\in\Chi(u)}\lambda_v e_v \right> \sum_{w\in\Chi(u)} \lambda_w e_w = \nonumber\\
  &= \frac 1{\|\Sl e_u\|^2} \left(\sum_{v\in\Chi(u)} \overline{\lambda_v} f(v) \right) \sum_{w\in\Chi(u)} \lambda_w e_w.
  \label{eq:Puf}
 \end{align}
 Observe that if $\left<P_u f, P_v g\right> \neq 0$ for some $u,v\in V$ and $f,g\in\ell^2(V)$, then from (\ref{eq:Puf}) it follows that there exists $w\in\Chi(u)\cap\Chi(v)$. This implies that $u=\pa(w)=v$. Hence the orthogonality of the sum in (iii) follows.

 For any $f\in\mathcal{D}(\Sl^*)$, $u\in V_{\boldsymbol{\lambda}}^+$ and $w\in V$ we infer from (\ref{eq:Puf}) that
 \begin{align}
  \allowdisplaybreaks
  (P_u f)(w) &= \begin{cases}\displaystyle \frac{\lambda_w}{\|\Sl e_u\|^2} \sum_{v\in\Chi(u)} \overline{\lambda_v}f(v),&\text{if }w\in\Chi(u)\\
						0,&\text{if }w\in V\setminus\Chi(u)
                       \end{cases} = \nonumber\\
 &= \begin{cases}\displaystyle \frac{\lambda_w}{\|\Sl e_{\pa(w)}\|^2} \sum_{v\in\Chi(\pa(w))} \overline{\lambda_v}f(v),&\text{if }w\in\Chi(u)\\
						0,&\text{if }w\in V\setminus\Chi(u)
                       \end{cases} = \nonumber\\
 &= \begin{cases}\displaystyle \frac { 1}{\|\Sl e_{\pa(w)}\|} (|\Sl^*|f)(w),&\text{if }w\in\Chi(u)\\
						0,&\text{if }w\in V\setminus\Chi(u)
                       \end{cases} .
  \label{eq:Pufw}
 \end{align}
 Hence for all $f\in\mathcal{D}(\Sl^*)$ and $w\in V^\circ$ we obtain
 \begin{equation*}
  (|\Sl^*|f)(w) = \|\Sl e_{\pa(w)}\| (P_{\pa(w)}f)(w) = \sum_{u\in V} \|\Sl e_u\| (P_u f)(w).
 \end{equation*}
 Therefore $|\Sl^*|f = \sum_{u\in V} \|\Sl e_u\| P_u f$
 for every $f\in\mathcal{D}(\Sl^*)$. To show that
 \begin{equation}
  |\Sl^*| = \bigoplus_{u\in V} \|\Sl e_u\| P_u,
  \label{eq:|Sl*|pu}
 \end{equation}
 it now suffices to check the inclusion
 \begin{equation*}
  \mathcal{D}\left(\bigoplus_{u\in V} \|\Sl e_u\| P_u\right) \subseteq \mathcal{D}(\Sl^*).
 \end{equation*}
 Let $f$ belong to the left-hand side. This means that
 \begin{eqnarray}
  \infty &>& \sum_{u\in V} \|\Sl e_u\|^2 \|P_u f\|^2 \stackrel{(\ref{eq:Puf})}{=} \nonumber\\
  &=& \sum_{u\in V} \|\Sl e_u\|^2 \left(\frac1{\|\Sl e_u\|^4} \left| \sum_{v\in\Chi(u)} \overline{\lambda_v}f(v) \right|^2 \left\| \Sl e_u\right\|^2 \right) =\nonumber\\
  &=& \sum_{u\in V} \left| \sum_{v\in\Chi(u)} \overline{\lambda_v}f(v) \right|^2,
  \label{eq:DOpl}
 \end{eqnarray}
 which implies that $f\in\mathcal{D}(\Sl^*)= \mathcal{D}(|\Sl^*|)$. This completes the prove of (\ref{eq:|Sl*|pu}).

 Let now $\alpha>0$ be arbitrary. The assertion (iii) of the theorem follows immediately from (\ref{eq:|Sl*|pu}). To prove (i) it now suffices to observe that, by calculations similar to (\ref{eq:DOpl}), $f\in \mathcal{D}(|\Sl^*|^\alpha) = \mathcal{D}(\bigoplus_{u\in V} \|\Sl e_u\|^\alpha P_u)$ if and only if
 \begin{equation*}
   \infty > \sum_{u\in V} \|\Sl e_u \|^{2\alpha} \|P_u f\|^2 = \sum_{u\in V} \|\Sl e_u\|^{2\alpha-2} \left|\sum_{v\in\Chi(u)} \overline{\lambda_v} f(v) \right|^2.
 \end{equation*}
 Finally, let $f\in\mathcal{D}(|\Sl^*|^\alpha)$ and $w\in V$. Then
 \begin{align*}
  \allowdisplaybreaks
  (|\Sl^*|^\alpha f)(w) &\stackrel{(iii)}{=} \sum_{u\in V} \|\Sl e_u\|^{\alpha} (P_u f)(w) =\\
  &=\sum_{u\in V_{\boldsymbol{\lambda}}^+} \|\Sl e_u\|^{\alpha} (P_u f)(w) =\\[1ex]
  &=\begin{cases} \|\Sl e_{\pa(w)}\|^\alpha (P_{\pa(w)} f)(w), &\text{if }w\in \Chi(V_{\boldsymbol{\lambda}}^+)\\
  								0, &\text{if } w\in V\setminus \Chi(V_{\boldsymbol{\lambda}}^+)
  								\end{cases} =\\[2ex]
  &\stackrel{(\ref{eq:Pufw})}{=} \begin{cases} \displaystyle\|\Sl e_{\pa(w)}\|^{2\alpha-2}\lambda_w \sum_{v\in\Chi(\pa(w))} \overline{\lambda_v}f(v),&\text{if }w\in\Chi(u)\\
						0,&\text{if }w\in V\setminus\Chi(u)
                       \end{cases},
 \end{align*}
 which is exactly the claim of (ii). Thus the prove is complete.
 \end{proof}

\section{Aluthge transform of a weighted shift}

In this section we give a description of the $t$-Aluthge transform of a weighted shift on a directed tree. It turns out that its closure is again a weighed shift on the same tree.

\begin{thm}
 Let $\Sl$ be a densely defined weighted shift on $\mathfrak{T}=(V,E)$ with $\boldsymbol{\lambda}:V^\circ\to\mathbb{C}$ and let $t\in(0,1]$. Then
 \begin{itemize}
  \item[(i)] $\mathcal{D}(\Delta_t(\Sl)) = \mathcal{D}(S_{\boldsymbol{\mu}})\cap\mathcal{D}(|\Sl|^{1-t})$, where
 \begin{equation}
  \mu_v = \begin{cases}\displaystyle
  					\frac{\|\Sl e_v\|^t}{\|\Sl e_{\pa(v)}\|^t} \lambda_v, & \text{if }\pa(v)\in V_{\boldsymbol{\lambda}}^+\\[2ex]
  					0, & \text{otherwise} \end{cases} ,\quad v\in V^\circ,
  \label{eq:mu}
 \end{equation}
  \item[(ii)] $\Delta_t(\Sl)$ is closable and $\overline{\Delta_t(\Sl)} = S_{\boldsymbol{\mu}}$.
 \end{itemize}

 \label{thm:DtSl}
\end{thm}

\begin{proof}
 Since $\Delta_t(\Sl) = |\Sl|^t S_{\boldsymbol{\pi}} |\Sl|^{1-t}$, where $\pi$ is given by (\ref{eq:pi}), for any \mbox{$f\in\ell^2(V)$} we have
 \begin{equation}
  f\in\mathcal{D}(\Delta_t(\Sl)) \iff \left(f\in\mathcal{D}(|\Sl|^{1-t})\text{ and }S_{\boldsymbol{\pi}}|\Sl|^{1-t}f\in\mathcal{D}(|\Sl|^t)\right)
  \label{eq:DiffDD}
 \end{equation}
 Let $f\in\mathcal{D}(|\Sl|^{1-t})$. Then, using (\ref{eq:mod}) and (\ref{eq:pi}), we obtain for every $v\in V^\circ$
 \begin{align}
  (S_{\boldsymbol{\pi}}|\Sl|^{1-t}f)(v) &= \pi_v (|\Sl|^{1-t}f)(\pa(v)) =\notag \\[1ex]
  &= \begin{cases} \displaystyle\frac{\lambda_v}{\|\Sl e_{\pa(v)}\|} \|\Sl e_{\pa(v)}\|^{1-t} f(\pa(v)), & \text{if }\pa(v)\in V_{\boldsymbol{\lambda}}^+\\[1ex]
  0, & \text{otherwise} \end{cases} = \notag\\[1ex]
    &= \begin{cases}\displaystyle \frac{\lambda_v}{\|\Sl e_{\pa(v)}\|^t} f(\pa(v)), & \text{if }\pa(v)\in V_{\boldsymbol{\lambda}}^+\\[1ex]
  0, & \text{otherwise} \end{cases} .
  \label{eq:SpSl}
 \end{align}
 From the above equation and Proposition \ref{prop:mod} it follows that $f\in\mathcal{D}(\Delta_t(\Sl))$ if and only if
 \begin{align*}
  \infty &> \sum_{v\in V} \|\Sl e_v\|^{2t}|(S_{\boldsymbol{\pi}}|\Sl|^{1-t}f)(v)|^2 =\\
  &= \sum_{v\in \Chi(V^+_\lambda)} \|\Sl e_v\|^{2t} \left|\frac{\lambda_v}{\|\Sl e_{\pa(v)}\|^t} f(\pa(v))\right|^2 =\\[1ex]
  &= \sum_{u\in V^+_\lambda} |f(u)|^2 \sum_{v\in\Chi(u)} \left| \frac{\|\Sl e_v\|^t}{\|\Sl e_{\pa(v)}\|^t} \lambda_v\right|^2 =\\[1ex]
  &= \sum_{u\in V} |f(u)|^2 \sum_{v\in\Chi(u)} |\mu_v|^2,
 \end{align*}
 which is equivalent to $f\in\mathcal{D}(S_{\boldsymbol{\mu}})$. This, due to (\ref{eq:DiffDD}), proves (i).

 Let now $f\in\mathcal{D}(\Delta_t(\Sl))$. Then, using (\ref{eq:mod}) and (\ref{eq:SpSl}), we obtain
 \begin{align*}
  (\Delta_t(\Sl)f)(v) &= (|\Sl|^t S_{\boldsymbol{\pi}} |\Sl|^{1-t} f)(v) =\\[1ex]
  &= \|\Sl e_v\|^t (S_{\boldsymbol{\pi}} |\Sl|^{1-t}f)(v) =\\[1ex]
  &= \begin{cases}\displaystyle
              \frac{\|\Sl e_v\|^t}{\|\Sl e_{\pa(v)}\|^t} \lambda_v f(\pa(v)), & \text{if }\pa(v)\in V_{\boldsymbol{\lambda}}^+\\[1ex]
              0, & \text{otherwise} \end{cases} =\\[1ex]
  &= \mu_v f(\pa(v)) = (S_{\boldsymbol{\mu}} f)(v),
 \end{align*}
 which proves that $\Delta_t(\Sl)\subseteq S_{\boldsymbol{\mu}}$. Hence $\Delta_t(\Sl)$ is closable and $\overline{\Delta_t(\Sl)}\subseteq S_{\boldsymbol{\mu}}$. But from Proposition \ref{prop:mod} we know that $\mathcal{E}_V\subseteq\mathcal{D}(|\Sl|^{1-t})$. Part (ii) follows now from (i) and Lemma \ref{lem:core}.
\end{proof}

\begin{cor}
 Let $\Sl$ be a weighted shift on a directed tree $\mathfrak{T}=(V,E)$ and let $t\in(0,1]$. Suppose there exists a constant $\alpha>0$ such that $\|\Sl e_u\|\geq\alpha$ for every $u\in V$. Then $\Delta_t(\Sl)=S_{\boldsymbol{\mu}}$, where $\mu$ is given by \eqref{eq:mu}.
\end{cor}
\begin{proof}
 According to Theorem \ref{thm:DtSl}, it suffices to show that $\mathcal{D}(S_{\boldsymbol{\mu}})\subseteq\mathcal{D}(|\Sl|^{1-t})$. Let $f\in\mathcal{D}(S_{\boldsymbol{\mu}})$. Then
 \begin{align*}
  \infty &> \sum_{u\in V} \left(\sum_{v\in\Chi(u)} |\mu_v|^2\right) |f(u)|^2 =\\
  &= \sum_{u\in V} \left(\sum_{v\in\Chi(u)} \frac{\|\Sl e_v\|^{2t}}{\|\Sl e_{\pa(v)}\|^{2t}}|\lambda_v|^2\right) |f(u)|^2 \geq\\
  &\geq \sum_{u\in V} \left(\sum_{v\in\Chi(u)} \frac{\alpha^{2t}}{\|\Sl e_u\|^{2t}}|\lambda_v|^2\right) |f(u)|^2 =\\[1ex]
  &= \alpha^{2t} \sum_{u\in V} \frac{\|\Sl e_u\|^2}{\|\Sl e_u\|^{2t}} |f(u)|^2
  =\alpha^{2t} \sum_{u\in V} \|\Sl e_u\|^{2-2t} |f(u)|^2,
 \end{align*}
 thus $f\in\mathcal{D}(|\Sl|^{1-t})$.
\end{proof}

\begin{rem}
 If $t=1$, then $\mathcal{D}(|\Sl|^{1-t}|)=\mathcal{D}(I) = \ell^2(V)$. Hence $\Delta_1(\Sl) = S_{\boldsymbol{\mu}}$, where $\boldsymbol{\mu}$ is given by (\ref{eq:mu}). This is in general not true for $t\in(0,1)$, which is shown by the following example.
\end{rem}

\begin{ex}
 Let $t\in(0,1)$ and $\mathfrak{T}=(V,E)$, where $V=\mathbb{N}=\{0,1,\ldots\}$ and $E=\{(n,n+1)\,:\,n\in\mathbb{N}\}$.
 For any $f\in\ell^2(\mathbb{N})$ such that $f(2k)\neq 0$ for each $k\in\mathbb{N}$, let
 \begin{equation*}
  \begin{array}{rcl}
   \lambda_{2k} &=& 0,\\
   \lambda_{2k+1} &=& |f(2k)|^{\frac1{t-1}}
  \end{array}
 \end{equation*}
 for all $k\in\mathbb{N}$.
 Then $\mu_n=0$ for every $n\in\mathbb{N}$ and therefore $\mathcal{D}(S_{\boldsymbol{\mu}})=\ell^2(V)$. But $f\notin\mathcal{D}(|\Sl|^{1-t})$, because
 \begin{multline*}
  \sum_{u\in V} \|\Sl e_u\|^{2-2t}|f(u)|^2 = \sum_{n=0}^\infty |\lambda_{n+1}|^{2-2t}|f(n)|^2 =\\
  =\sum_{k=0}^\infty |\lambda_{2k+1}|^{2-2t}|f(2k)|^2 = \sum_{k=0}^\infty |f(2k)|^{\frac{2-2t}{t-1}}|f(2k)|^2
  = \sum_{k=0}^\infty 1 = \infty.
 \end{multline*}
 Hence $\Delta_t(\Sl) \subsetneq S_{\boldsymbol{\mu}}$. This example shows in particular that $\Delta_t(\Sl)$ may not be closed.
\end{ex}

\section{Aluthge transform of $\Sl^*$}

The following theorem provides a formula for the $t$-Aluthge transform of the adjoint of a weighted shift.

\begin{thm}
  Let $\Sl$ be a densely defined weighted shift on a directed tree $\mathfrak{T}=(V,E)$ and let $t\in(0,1]$. Then $\mathcal{E}_V\subseteq \mathcal{D}(\Delta_t(\Sl^*))$ and
  \begin{equation*}
  \Delta_t(\Sl^*)e_v = \begin{cases} \displaystyle\overline{\lambda_v}\frac {|\pi_{\pa(v)}|^2}{\mu_{\pa(v)}} \Sl e_{\pa^2(v)}& \text{if }v\in\Chi^2(V_{\boldsymbol{\lambda}}^+)\\[1ex]
   0&\text{if }v\in V\setminus\Chi^2(V_{\boldsymbol{\lambda}}^+) \end{cases},
  \end{equation*}
  where $\boldsymbol{\mu}=\{\mu_w\}_{w\in V^\circ}$ and $\boldsymbol{\pi}=\{\pi_w\}_{w\in V^\circ}$ are given by \eqref{eq:mu} and \eqref{eq:pi} respectively.
  \label{thm:DSl*}
\end{thm}

\begin{proof}
 Let $u,v\in V$ be any vertices and let $P_u$ be as in Theorem \ref{thm:S*a}. Then, by \eqref{eq:Puf},
 \begin{equation*}
  P_u e_v = \begin{cases}\displaystyle \frac{\overline{\lambda_v}}{\|\Sl e_u\|^2} \displaystyle\sum_{w\in\Chi(u)}\lambda_we_w,&\text{if }v\in\Chi(u)\\
  0&\text{if }v\in V\setminus\Chi(u). \end{cases}.
 \end{equation*}
 Hence, from Theorem \ref{thm:S*a} (iii) it follows for every $\alpha>0$ that $e_v\in\mathcal{D}(|\Sl^*|^\alpha)$ and the following equality holds:
 \begin{equation}
  |\Sl^*|^\alpha e_v = \begin{cases}\displaystyle
                   \frac{\overline{\lambda_v}}{\|\Sl e_{\pa(v)}\|^{2-\alpha}} \sum_{w\in\Chi(\pa(v))}\lambda_we_w,   &\text{if }v\in V^\circ,\\
                   0,&\text{if }v=\rot. \end{cases}.
   \label{eq:Sl*aev}
 \end{equation}

 Let now $\Sl^* = U|\Sl^*|$ be the polar decomposition of $\Sl^*$. Then, by Proposition \ref{prop:Spi}, $U=S_{\boldsymbol{\pi}}^*$, where $\boldsymbol{\pi}=\{\pi_u\}_{u\in V}$ is given by \eqref{eq:pi}. From Proposition \ref{prop:Sl*} it follows that for every $w\in V^\circ$
 \begin{equation}
  S_{\boldsymbol{\pi}}^* e_w = \overline{\pi_w} e_{\pa(w)} = \frac{\overline{\lambda_w}}{\|\Sl e_{\pa(w)}\|}e_{\pa(w)}.
  \label{eq:Spew}
 \end{equation}
 Take $u\in V$. From \eqref{eq:Spew} we obtain
 \begin{align*}
  \sum_{w\in\Chi(u)} \|\lambda_w S_{\boldsymbol{\pi}}^*e_w\|^2 &= \sum_{w\in\Chi(u)} \frac{|\lambda_w|^4}{\|\Sl e_u\|^2}
  \leq \left(\sum_{w\in\Chi(u)} \frac{|\lambda_w|^2}{\|\Sl e_u\|} \right)^2 =\nonumber\\
  &= \left(\frac {\|\Sl e_u\|^2}{\|\Sl e_u\|}\right)^2 = \|\Sl e_u\|^2
 \end{align*}
 Hence the series $\sum_{w\in\Chi(u)} \lambda_w S_{\boldsymbol{\pi}}^*e_w$ is convergent in $\ell^2(V)$ and by (\ref{eq:Spew})
 \begin{eqnarray}
  \sum_{w\in\Chi(u)} \lambda_w S_{\boldsymbol{\pi}}^*e_w &=&
   \sum_{w\in\Chi(u)} \frac {|\lambda_w|^2}{\|\Sl e_{\pa(w)}\|} e_{\pa(w)}=\notag\\
   &=&\sum_{w\in\Chi(u)} \frac {|\lambda_w|^2}{\|\Sl e_u\|} e_u = \|\Sl e_u\| e_u.
   \label{eq:szereg}
 \end{eqnarray}
 Since the series $\sum_{w\in\Chi(u)} \lambda_w e_w = \Sl e_u$ is also convergent and $S_{\boldsymbol{\pi}}^*$ is a closed operator, it follows from (\ref{eq:Sl*aev}) and (\ref{eq:szereg}) that $|\Sl^*|^te_v\in\mathcal{D}(S_{\boldsymbol{\pi}}^*)$ for any $v\in V^\circ$ and
 \begin{align*}
  S_{\boldsymbol{\pi}}^*|\Sl^*|^{1-t} e_v &= \frac{\overline{\lambda_v}}{\|\Sl e_{\pa(v)}\|^{2-(1-t)}} S_{\boldsymbol{\pi}}^*\left(\sum_{w\in\Chi(\pa(v))}\lambda_we_w \right) =\\
  &=\frac{\overline{\lambda_v}}{\|\Sl e_{\pa(v)}\|^{1+t}} \sum_{w\in\Chi(\pa(v))}\lambda_w S_{\boldsymbol{\pi}}^* e_w  =\\
  &=\frac{\overline{\lambda_v}}{\|\Sl e_{\pa(v)}\|^{t}} e_{\pa(v)}
 \end{align*}
 Using (\ref{eq:Sl*aev}) again, we get for any $v\in\Chi(V^\circ)$
 \begin{multline*}
  |\Sl^*|^tS_{\boldsymbol{\pi}}^*|\Sl^*|^{1-t} e_v =\\
  = \frac{ \overline{\lambda_v}}{\|\Sl e_{\pa(v)}\|^{t}}\cdot \frac{\overline{\lambda_{\pa(v)}}}{ \|\Sl e_{\pa^2(v)}\|^{2-t}} \sum_{w\in\Chi(\pa^2(v))} \lambda_w e_w =\\
  = \overline{\lambda_v}\frac {|\pi_{\pa(v)}|^2}{\mu_{\pa(v)}} \Sl e_{\pa^2(v)}
 \end{multline*}
 and clearly $|\Sl^*|^{1-t}S_{\boldsymbol{\pi}}^*|\Sl^*|^t e_v=0$ for every $v\in \Chi(\rot)\cup\{\rot\}$.
\end{proof}

\section{An example of an operator with trivial Aluthge transform}

In this section we construct a weighted shift $\Sl$ with the following properties: $\Sl$ is densely defined, injective and hyponormal, while $\mathcal{D}(\Delta_t(\Sl))=\{0\}$ for every $t\in(0,1]$ and $\Delta_t(\Sl^*)$ is not closable for any $t\in(0,1)$. We also show that such an example can be constructed in the class of composition operators.
\bigskip

For any sequence $v: \mathbb{Z}\to\mathbb{Z}_+\cup\{-1\}$ we define
\begin{align*}
 m(v)&:= \inf\{n\in\mathbb{Z}\,:\,v_n\neq0\}-2, \\
 M(v)&:= \sup\{n\in\mathbb{Z}\,:\,v_n>-1\}.
\end{align*}
Let
\begin{multline}
 V=\{ v:\mathbb{Z}\to\mathbb{Z}_+\cup\{-1\}\,:\,m(v)>-\infty, M(v)<\infty\\
  \text{ and }v_n>-1\text{ for }n\leq M(v)\},
 \label{eq:V}
\end{multline}
\begin{equation}
 E=\{ (u,v)\in V\times V\,:\,M(v)=M(u)+1\text{ and }u_n=v_n\text{ for }n\leq M(u)\}.
 \label{eq:E}
\end{equation}
Then $\mathfrak{T}=(V,E)$ is a rootless directed tree, such that for every $u\in V$ the set $\Chi(u)$ is countable. Vertices of $V$ are sequences of the form
 $$u=(\ldots,0,0,u_{m(u)},\ldots,u_{M(u)-1},u_{M(u)},-1,-1,\ldots)$$
  and for a vertex $u$ given by the above formula we have
 \begin{align*}
  \pa(u) &= (\ldots,0,u_{m(u)},\ldots,u_{M(u)-1},-1,-1,-1,\ldots),\\
  \Chi(u) &= \{ (\ldots,0,u_{m(u)},\ldots,u_{M(u)-1},u_{M(u)},n,-1,\ldots)\,:\,n\in\mathbb{Z}_+ \}.
 \end{align*}

For any $v=(\ldots,0,v_{m(v)},\ldots,v_{M(v)},-1,\ldots)\in V$ let
\begin{equation}
 \lambda_v= \frac{ 2^{v_{m(v)}+\ldots+v_{M(v)-1}} }{v_{M(v)}+1}
 \label{eq:lambda}
\end{equation}

In this section $\Sl$ will always stand for the weighted shift on $\mathfrak{T}$ with weights given by (\ref{eq:lambda}).
\bigskip

We start by proving that $\Sl$ is densely defined. This follows from Proposition~\ref{prop:Sl} and the following:

\begin{prop}
 For every $u\in V$, $e_u\in\mathcal{D}(\Sl)$ and
 \begin{equation*}
  \|\Sl e_u\| = 2^{u_{m(u)}+\ldots+u_{M(u)}}\gamma,
 \end{equation*}
 where $\gamma=\left(\sum_{n=1}^\infty n^{-2}\right)^{\frac12}$.
 \label{prop:Sleu}
\end{prop}
\begin{proof}
 From \eqref{eq:lambda} we get
 \begin{align*}
  \sum_{v\in\Chi(u)} |\lambda_v|^2 &= \sum_{v\in\Chi(u)} \frac {2^{2(v_{m(v)}+\ldots+v_{M(v)-1})}}{(v_{M(v)}+1)^2}=\\
  &= \sum_{v\in\Chi(u)} \frac {2^{2(u_{m(u)}+\ldots+u_{M(u)})}}{(v_{M(v)}+1)^2} = 2^{2(u_{m(u)}+\ldots+u_{M(u)})}\gamma^2.
 \end{align*}
 The claim follows now from Proposition \ref{prop:Sl} (ii).
\end{proof}

To show hyponormality of $\Sl$ we use Theorem \ref{thm:hyp}.

\begin{prop}
 Operator $\Sl$ is hyponormal.
 \label{prop:Slhyp}
\end{prop}
\begin{proof}
 From Proposition \ref{prop:Sleu} it follows that $\|\Sl e_v\|>0$ for every $v\in V$, so \eqref{eq:hyp1} is satisfied trivially.
 As for \eqref{eq:hyp}, for any $u\in V$ we have $\Chi_\lambda^+(u) = \Chi(u)$ and
 \begin{align*}
  \sum_{v\in\Chi(u)} \frac{|\lambda_v|^2}{\|\Sl e_v\|^2} &= \sum_{v\in\Chi(u)}
   \frac{2^{2(v_{m(v)}+\ldots+v_{M(v)-1})}}{(v_{M(v)}+1)^2\cdot 2^{2(v_{m(v)}+\ldots+v_{M(v)})} \gamma^2} =\\
   &= \frac1{\gamma^2} \sum_{v\in\Chi(u)} \frac1{(v_{M(v)}+1)^2 2^{2v_{M(v)}}} < 1,
 \end{align*}
 because
 \begin{equation*}
  \sum_{v\in\Chi(u)} \frac1{(v_{M(v)}+1)^2 2^{2v_{M(v)}}} < \sum_{v\in\Chi(u)} \frac1{(v_{M(v)}+1)^2} = \gamma^2.
 \end{equation*}
 This completes the proof.
\end{proof}

 In turn we show that the Aluthge transform of $\Sl$ has trivial domain. Moreover, $t$-Aluthge transform of $\Sl$ has trivial domain for arbitrarily small $t$.

\begin{prop}
 For any $t\in(0,1]$ the domain of $\Delta_t(\Sl)$ is $\{0\}$.
 \label{prop:trivial}
\end{prop}
\begin{proof}
 Let $t\in(0,1]$. From Theorem \ref{thm:DtSl} and Proposition \ref{prop:Sleu} we get \mbox{$\Delta_t(\Sl)\subseteq S_{\boldsymbol{\mu}}$}, where
 \begin{align}
  \mu_v &= \frac{\|\Sl e_v\|^t}{\|\Sl e_{\pa(v)}\|^t}\lambda_v =\notag\\
  &= \frac{2^{t(v_{m(v)}+\ldots+v_{M(v)})}\gamma^t}{2^{t(v_{m(v)}+\ldots+v_{M(v)-1})}\gamma^t}\cdot \frac{2^{v_{m(v)}+\ldots+v_{M(v)-1}}}{v_{M(v)}+1}=\notag\\
   &= \frac {2^{v_{m(v)}+\ldots+v_{M(v)-1}+tv_{M(v)}}}{v_{M(v)}+1}.
   \label{eq:muv}
 \end{align}
 Hence for any $u\in V$ we have
 \begin{align*}
  \sum_{v\in\Chi(u)} |\mu_v|^2 &= \sum_{v\in\Chi(u)} \frac{ 2^{2(v_{m(v)}+\ldots+v_{M(v)-1}+tv_{M(v)})}}{(v_{M(v)}+1)^2} =\\
  &= 2^{2(u_{m(u)}+\ldots+u_{M(u)})} \sum_{v\in\Chi(u)} \frac{2^{2tv_{M(v)}}}{(v_{M(v)}+1)^2} = \infty,
 \end{align*}
 and therefore, by Proposition \ref{prop:Sl}, $e_u\notin\mathcal{D}(S_{\boldsymbol{\mu}})$. The claim follows now from Lemma \ref{lem:core}.
\end{proof}

The fact that $\Delta_t(\Sl^*)$ is not closable will follow from the lemma below:

\begin{lem}
 For any $t\in(0,1)$ the operator $\Delta_t(\Sl^*)$ is densely defined and
 \begin{equation*}
 \mathcal{D}(\Delta_t(\Sl^*)^*) = \mathcal{N}(\Delta_t(\Sl^*)^*) = \mathcal{N}(\Sl^*).
\end{equation*}
\end{lem}

\begin{proof}
 By Theorem \ref{thm:DSl*}, $\mathcal{E}_V\subseteq \mathcal{D}(\Delta_t(\Sl^*))$ and obviously $\mathcal{E}_V$ is dense in $\ell^2(V)$. Moreover, since $\Chi^2(V_{\boldsymbol{\lambda}}^+)=V$, we have for every $v\in V$
 \begin{equation}
  \Delta_t(\Sl^*) e_v = \overline{\lambda_v} \frac{|\pi_{\pa(v)}|^2}{\mu_{\pa(v)}}\Sl e_{\pa^2(v)}.
  \label{eq:DtSl*}
  \end{equation}
 Let $v=(\ldots,0,v_{m(v)},\ldots,v_{M(v)},-1,\ldots)$. From \eqref{eq:lambda} and \eqref{eq:muv} we obtain
 \begin{align}
  \frac{\overline{\lambda_v}}{\mu_{\pa(v)}} &= \frac{2^{v_{m(v)}+\ldots+v_{M(v)-1}}}{v_{M(v)}+1}
     \frac{v_{M(v)-1}+1}{2^{v_{m(v)}+\ldots+v_{M(v)-2}+tv_{M(v)-1}}} =\notag\\
  &= \frac{v_{M(v)-1}+1}{v_{M(v)}+1} 2^{(1-t)v_{M(v)-1}}. \label{eq:frac}
 \end{align}
 In turn, by \eqref{eq:pi} and \eqref{eq:Sleu} we have
 \begin{align}
  |\pi_{\pa(v)}|^2 &= \frac{|\lambda_{\pa(v)}|^2}{\|\Sl e_{\pa^2(v)}\|^2} =\notag\\
  &= \frac {2^{2(v_{m(v)}+\ldots+v_{M(v)-2})}}{(v_{M(v)-1}+1)^2 2^{2(v_{m(v)}+\ldots+v_{M(v)-2})}\gamma^2}=\notag\\
  &= \frac1 {(v_{M(v)-1}+1)^2\gamma^2}. \label{eq:piparv}
 \end{align}
 Combining \eqref{eq:DtSl*}, \eqref{eq:frac} and \eqref{eq:piparv} leads to the equality
 \begin{equation*}
  \Delta_t(\Sl^*)e_v = \frac{2^{(1-t)v_{M(v)-1}}}{(v_{M(v)-1}+1)(v_{M(v)}+1)\gamma^2} \sum_{w\in\Chi(\pa^2(v))} \lambda_w e_w.
 \end{equation*}

 Let $f\in\mathcal{D}(\Delta_t(\Sl^*)^*)$. Then for any $v\in V$
 \begin{align}
  (\Delta_t(\Sl^*)^*f)(v) &= \left< \Delta_t(\Sl^*)^*f,e_v\right> = \left< f, \Delta_t(\Sl^*)e_v\right> =\notag\\
  &= \frac{2^{(1-t)v_{M(v)-1}}}{(v_{M(v)-1}+1)(v_{M(v)}+1)\gamma^2} \sum_{w\in\Chi(\pa^2(v))} \overline{\lambda_w} f(w)=\notag\\
  &= \frac{2^{(1-t)v_{M(v)-1}}}{(v_{M(v)-1}+1)(v_{M(v)}+1)\gamma^2} (\Sl^* f)(\pa^2(v)).
  \label{eq:DtSl**}
 \end{align}
 This gives the inclusion $\mathcal{N}(\Sl^*) \subseteq \mathcal{N}(\Delta_t(\Sl^*)^*)$. It suffices to show that $\mathcal{D}(\Delta_t(\Sl^*)^*)\subseteq \mathcal{N}(\Sl^*)$.

 Suppose there exists $f\in\mathcal{D}(\Delta_t(\Sl^*)^*)$ such that $\Sl^*f\neq0$. Let $$u=(\ldots,0,u_{m(u)},\ldots,u_{M(u)},-1,\ldots)\in V$$ be such that $(\Sl^*f)(u) \neq0$. Let $v^{(k)} = (\ldots,0,u_{m(u)},\ldots,u_{M(u)},k,0,-1,\ldots)$ for every $k\in\mathbb{Z}_+$. Then $\pa^2(v^{(k)})=u$ and
 \begin{align*}
  \|\Delta_t(\Sl^*)^*f\|^2 &\geq \sum_{k=0}^\infty |(\Delta_t(\Sl^*)^*f)(v^{(k)})|^2 =\\
  &\stackrel{(\ref{eq:DtSl**})}{=} \sum_{k=0}^\infty \frac{2^{2(1-t)k}}{(k+1)^2\gamma^2} \left|(\Sl^*f)(u)\right|^2 = \infty,
 \end{align*}
 because $t<1$. This is a contradiction. Thus $\Sl^*f=0$, which completes the proof.
\end{proof}

\begin{cor}
Operator $\Delta_t(\Sl^*)$ is not closable for any $t\in(0,1)$.
 \label{cor:closable}
\end{cor}
\begin{proof}
 Since $\Sl^*$ is a non-zero closed operator, $\mathcal{D}(\Delta_t(\Sl^*)^*)=\mathcal{N}(\Sl^*)$ is not dense in $\ell^2(V)$, which completes the proof.
\end{proof}

 By \cite[Lemma 4.3.1]{jjs3}, every weighted shift on a rootless directed tree with nonzero weights is unitarily equivalent to a composition operator in an $L^2$-space over a $\sigma$-finite measure. From this, together with Propositions \ref{prop:Sleu}, \ref{prop:Slhyp}, \ref{prop:trivial} and Corollary \ref{cor:closable}, we obtain the following theorem.

\begin{thm}
 There exists a hyponormal composition operator $C$ in an \mbox{$L^2$-space} over a $\sigma$-finite measure such that $\mathcal{D}(\Delta_t(C))=\{0\}$ for $t\in(0,1]$ and $\Delta_t(C^*)$ is not closable for $t\in(0,1)$.
\end{thm}

\begin{rem}
 For any $u\in V$ let $W:=\Des(u) = \bigcup_{n=0}^\infty \Chi^n (u)$ and let $\boldsymbol{\lambda}'=\{\lambda_v\}_{v\in W\setminus\{u\}}$. Then $S_{\boldsymbol{\lambda}'}$ is a weighted shift on a directed tree with root $u$. Moreover, $S_{\boldsymbol{\lambda}'}$ has all properties claimed for $\Sl$, i.e. $S_{\boldsymbol{\lambda}'}$ is densely defined, injective and hyponormal, its $t$-Aluthge transform has trivial domain for $t\in(0,1]$ and $t$-Aluthge transform of $S_{\boldsymbol{\lambda}'}^*$ is not closable for $t\in (0,1)$. These assertions can be shown by repeating the proofs of all results from this section with appropriate changes.
 \label{rem:root}
\end{rem}

It turns out that the tree given by \eqref{eq:V} and \eqref{eq:E} and the one described by Remark \ref{rem:root} are the only directed trees on which such an example can be constructed. This fact is stated in the following proposition.

\begin{prop}
 Let $\mathfrak{T}=(V,E)$ and $\lambda=\{\lambda_u\}_{u\in V^\circ}\subseteq\mathbb{C}\setminus\{0\}$. Suppose the weighted shift $\Sl$ is densely defined and $\mathcal{D}(\Delta_t(\Sl))=\{0\}$ for some $t\in(0,1]$. Then $\#\Chi(u) =\aleph_0$ for every $u\in V$.
\end{prop}

\begin{proof}
 Let $u\in V$.
 Due to Proposition \ref{prop:Sl}, $\overline{\mathcal{D}(\Sl)}=\ell^2(V)$ implies that for $u\in V$ we have $e_u\in\mathcal{D}(\Sl)$ and
 $$\|\Sl e_u\|^2 = \sum_{v\in\Chi(u)} |\lambda_v|^2.$$
 Since $|\lambda_v|^2>0$ for every $v\in\Chi(u)$ and the above series is convergent, it follows that $\#\Chi(u)\leq\aleph_0$.

 Let $t\in(0,1]$ be such that $e_u\notin \mathcal{D}(\Delta_t(\Sl))$. By Theorem \ref{thm:DtSl}, $\mathcal{D}(\Delta_t(\Sl))=\mathcal{D}(S_{\boldsymbol{\mu}})\cap\mathcal{D}(|\Sl|^{1-t})$, where $\mu$ is given by \eqref{eq:mu}. Since $e_u\in \mathcal{D}(\Sl)\subseteq\mathcal{D}(|\Sl|^{1-t})$, it follows that $e_u\notin \mathcal{D}(S_{\boldsymbol{\mu}})$. Hence
 \begin{equation*}
  \infty = \sum_{w\in V}\left(\sum_{v\in\Chi(w)} |\mu_v|^2\right)|e_u(w)|^2
  = \sum_{v\in\Chi(u)} |\mu_v|^2,
 \end{equation*}
 which is possible only if $\#\Chi(u)\geq\aleph_0$. This completes the proof.
\end{proof}

A similar result with $\Sl^2$ instead of $\Delta_t(\Sl)$ was obtained in \cite{jjs2}.

\subsection*{Acknowledgements}

I would like to thank my supervisor, prof. Jan Stochel for encouragement and motivation, as well as substantial help he provided me while working on this paper.

\end{document}